\documentclass{article}
\usepackage{amsmath,amssymb,amsthm,stmaryrd,makeidx}
\usepackage[all]{xy}

\DeclareMathOperator{\id}{id}

\DeclareMathOperator{\hmm}{Hom}

\DeclareMathOperator{\ph}{Ph}

\DeclareMathOperator{\pts}{Pts}

\newcommand{\grmcat}[1]

\input xypic

\newtheorem{proposition}{Proposition}
\newtheorem{theorem}{Theorem}
\newtheorem{lemma}{Lemma}

\newtheorem{definition}{Definition}

\newtheorem{example}{Example}

\makeindex

\begin{document}
\author{Arvid Siqveland}
\title{Connections on Associative Varieties}

\maketitle

\begin{abstract} We have embedded the category of smooth real $n$-manifolds faithfully into the category of associative varieties, \cite{S2603}. Many constructions on smooth manifolds can then be generalized to associative varieties. Here we will generalize the well known interpretation of connections on manifolds to connections on associative varieties. This eventually proves the existence of geodesics in an associative variety.
\end{abstract}

\section{Introduction}
In the book \cite{Laudal21}, O. A. Laudal propose a model for the universe as an associative real variety with a noncommutative tangent variety. In the book \cite{S23}, the author gives the definition of associative varieties in general. This is made precise in the preprint \cite{S2603} where we define real associative varieties with their tangent varieties, and define a Riemannian geometry on these tangent varieties. With the concept of connections and geodesics which we will define here, we can go back to Laudal's book \cite{Laudal21}, Mathematical Models in Science, and use the geodesics in a chosen Riemannian metric to define distance, which implies time. The geometrical interpretation follows the book, Introduction to Riemannian Manifolds, by John M. Lee, \cite{Lee18}.

We define schemes categorically as the gluing of affine varieties, \cite{S2603}, obtaining the advantage with real associative varieties, that we can work locally.
We use the notation $\mathbb R\langle n\rangle=\mathbb R\langle x_1,\dots, x_n\rangle$ for the free noncommutative polynomial algebra in $n$ variables, and $\mathbb R[n]=\mathbb R[x_1,\dots,x_n],$ when it is commutative. We recall from \cite{S2603} that for an $\mathbb R$-algebra $A,$ the set $\pts_{\mathbb R}A=\hmm_{\mathbb R}(A,\mathbb R),$ which comes with a structure of associative variety, and that a vector bundle $$\pi:E\rightarrow X$$ of rank $k$ is defined locally as an $\mathbb R$-algebra homomorphism $$\phi:A\rightarrow (A\otimes_{\mathbb R}\mathbb R\langle x_1,\dots,x_m\rangle)/I=A\langle m\rangle/I$$ where $I$ is a two-sided ideal, and such that for each point $p:A\rightarrow\mathbb R$ we have $\mathbb R\otimes_A A\langle m\rangle/I\simeq\mathbb R\langle k\rangle.$ A vector field in the vector bundle $$\pi:E=\pts_{\mathbb R}(A\langle m\rangle/I)\rightarrow X=\pts_{\mathbb R}A$$ is a section $\sigma^\ast:X\rightarrow E,$ i.e. such that $\pi\circ\sigma^\ast=\id,$ corresponding to a section $\sigma:A\langle m\rangle/I\rightarrow A$ such that $\phi\circ\sigma=\id.$ Here $\phi:A\rightarrow A\langle m\rangle/I$ is the canonical homomorphism. We denote the set of vector fields in $E$ over $X$ by $\Gamma(E).$

\begin{lemma} Let $E\rightarrow X$ be a $k$-bundle. Then $\Gamma(E)$ is the set of sections $$\sigma:X\rightarrow\mathbb R^n,$$ and has as such a structure of a real $k$-dimensional vector space, defined by point-wise addition and scalar multiplication.
\end{lemma}

\begin{proof} For each point $p:A\rightarrow\mathbb R$ in $X=\pts_{\mathbb R}A,$ the section $\sigma: A\langle m\rangle/I\rightarrow A$ defines an $\mathbb R$-algebra homomorphism $$\sigma(p):\mathbb R\langle k\rangle=\mathbb R\otimes_A A\langle m\rangle/I\rightarrow\mathbb R=\mathbb R\otimes_A A.$$ This proves the lemma because the set of algebra homomorphisms $\mathbb R\langle k\rangle\rightarrow\mathbb R$ is in one-to-one correspondence with the points in $\mathbb R^k.$  
\end{proof}

In the remaining sections, following Lee \cite{Lee18}, we use $M$ as the name for a general associative variety over $\mathbb R,$ and reserve $X,Y$ as names for sections of vector bundles.

\section{Lines and Curves in Real Associative Varieties}

A curve in an associative $\mathbb R$-variety corresponds (locally affine) to a homomorphism $$\gamma:A\rightarrow\mathbb R[t].$$ In particular, let $A=\mathbb R[n].$ Then $\pts_{\mathbb R}A=\mathbb R^n$ and a curve in $\mathbb R^n$ is a homomorphism \begin{equation}\label{curveeq}\gamma:A\rightarrow\mathbb R[t]\end{equation} determined by $\gamma(x_i)=\gamma_i(t), 1\leq i\leq n.$ Thus all the geometric interpretations in Lee's book \cite{Lee18} follows immediately. 
In \cite{S2603} we defined the tangent variety $T\mathbb R^n$ of $\mathbb R^n$ as $T\mathbb R^n=\pts_{\mathbb R}(\ph A)$ where $$\ph(A)=\mathbb R[x_1,\dots,x_n,dx_1,\dots,dx_n]/E,$$ and where $E$ is the ideal determined by the property   that $d:A\rightarrow\ph A$ is an $\mathbb R$-derivation. Also, let $\pi:A\rightarrow\ph A$ be the homomorphism given by $\pi(x_i)=x_i,1\leq i\leq n.$
The tangent variety in a point $p=(a_1,\dots,a_n)\in\mathbb R^n$ is then $\pts_{\mathbb R}(\ph A|_p),$ with $$\ph A|_p=\mathbb R[dx_1|_p,\dots,dx_n|_p]=\ph A/(x_1-a_1,\dots,x_n-a_n).$$

For the curve $\gamma$ in (\ref{curveeq}), we have the well defined velocity vector $v(t)$ and the acceleration vector $a(t),$ for every $t,$ computed by differentiating the components: $$ v(t)=\gamma'(t)=\gamma_1'(t)dx_1|_{\gamma(t)}+\cdots+\gamma_n'(t)dx_n|_{\gamma(t)}\in \ph A|_{\gamma(t)},$$
$$a(t)=v'(t)=\gamma_1''(t)dx_1|_{\gamma(t)}+\cdots\gamma_n''(t)dx_n|_{\gamma(t)}\in  \ph A|_{\gamma(t)}.$$ A curve $\gamma(t)$ is a straight line if and only if $a(t)=0.$

\section{Connections}

We turn to the general affine situation, with $A$ an associative finitely generated $\mathbb R$-algebra and $M=\pts_{\mathbb R}A.$ We have defined the tangent variety of $M$ as the $\mathbb R$-variety $TM=\pts_{\mathbb R}(\ph(A)).$

\begin{definition} Let $M$ be an associative $\mathbb R$-variety, and let $p\in U\subseteq M$ with $U=\pts_{\mathbb R} A.$  Then the tangent space $T_pM$ in the point $p\in M$ is defined as $$T_p M=\mathbb R\otimes_A\ph(A)$$ where $\mathbb R$ is an $A$-module by the $\mathbb R$-algebra homomorphism $p:A\rightarrow\mathbb R.$ 
\end{definition}

We give a guiding example of directional derivatives.

\begin{example}\label{dirdereks} Let $A=\mathbb R[n].$ Then $$\ph(A)=\mathbb R[x_1,\dots,x_n,dx_1,\dots,dx_n]/(\{d(fg)-f(dg)-(df)g)\}_{f,g\in\mathbb R[n]}),$$ and for a point $p:A\rightarrow\mathbb R$ we find $T_p\mathbb R^n=\mathbb R[dx_1,\dots,dx_n].$ Together with the derivation $d:A\rightarrow\ph(A),$ we get the map $$d_p=p\otimes d:A\rightarrow\mathbb R\otimes_A\ph(A)$$ such that $$d_p(f)=\frac{df(p)}{dx_1}dx_1+\cdots+\frac{df(p)}{dx_n}dx_n=\nabla f(p).$$ Given a vector $v=(\alpha_1,\dots,\alpha_n)\in T_p\mathbb R^n,$ the Euclidean directional derivative of $f\in\mathbb R[n]$ is defined as the vector $$\overline\nabla_v f(p)=\alpha_1\frac{df(p)}{dx_1}dx_1+\cdots+\alpha_n\frac{df(p)}{dx_n}dx_n\in T_p\mathbb R^n.$$ Considering $f$ as a function $f:\mathbb R^n\rightarrow \mathbb R,$ or replacing $\mathbb R[n]$ with $C^\infty(\mathbb R^n),$ and evaluating $d(f)\in\ph(A)$ in the point $(p,\frac{1}{|v|}v)$ we find  $$d(f)(p,v)=v\cdot\nabla f(p),$$ the ordinary directional derivation of $f$ in the direction $v.$ The normalization of the vector $v$ is the reason for introducing a Riemannian metric on $\ph\mathbb R[n].$
\end{example}

\begin{definition}
For a point $p\in\pts_{\mathbb R} A$ and a vector $v=(\alpha_1,\dots,\alpha_n)\in\pts_{\mathbb R}(T_pM),$ we define the directional derivative of $f\in A$ as $$\overline\nabla_v f(p)=\alpha_1\frac{df(p)}{dx_1}dx_1+\cdots+\alpha_n\frac{df(p)}{dx_n}dx_n\in T_pM.$$
\end{definition}

Given a vector field $Y\in\mathfrak X(\mathbb R^n)=\Gamma(T\mathbb R^n).$ By definition, such a vector field can be seen as a map $\mathbb R^n\rightarrow\mathbb R^n,$ and we can consider its component functions $Y^i:\mathbb R^n\rightarrow\mathbb R.$ Locally, the vector field $Y$ corresponds to a section $\sigma:\ph\mathbb R[n]\rightarrow\mathbb R[n]$ and the component functions corresponds to $$y_i=\sigma(dx_i)\in\mathbb R[n].$$

For a vector $v\in T_p\mathbb R^n$ we can compute the directional derivatives of the component functions as in Example \ref{dirdereks}, and copy Lee's definition of the \emph{Euclidean directional derivative of $Y$ in the direction $v$} by the formula $$\overline\nabla_vY=\overline\nabla_v(y_1)dx_1|_p+\cdots\overline\nabla_v(y_n)dx_1|_p\in\ph\mathbb R^n|_p.$$ If $X$ is another vector field on $\mathbb R^n,$ we can evaluate $\overline\nabla_{X_p}Y$ in every point $p\in\mathbb R^n$  obtaining a new vector field $$\overline\nabla_XY=\overline\nabla_X(y_1)dx_1+\cdots\overline\nabla_X(y_n)dx_1\in\mathfrak X(\mathbb R^n).$$

\begin{definition} For vector fields $X,Y$ on $\pts_{\mathbb R}A=M$ we define the vector field $\overline\nabla_XY$ by $$\overline\nabla_XY=\overline\nabla_X(y_1)dx_1+\cdots\overline\nabla_X(y_n)dx_1\in\mathfrak X(M).$$
\end{definition}

As when we generalized differentiation as the application of a derivation, we generalize directional differentiation as an application of a connection. 

\begin{definition} Let $\pi:E\rightarrow M$ be a vector bundle over an associative variety $M.$ A connection in $E$ is a map $$\nabla:\mathfrak X(M)\times\Gamma(E)\rightarrow\Gamma(E),$$ written $(X,Y)\mapsto\nabla_XY,$ satisfying the following properties for an affine covering $\pi:E=\pts_{\mathbb R} B\rightarrow M=\pts_{\mathbb R}A.$
\begin{itemize}
\item[(i)] $\nabla_XY$ is linear over $A$ in $X$: for $f_1,f_2\in A$ and $X_1,X_2\in\mathfrak X(M),$ $$\nabla_{f_1X_1+f_2X_2}Y=f_1\nabla_{X_1}Y+f_2\nabla_{X_2}Y.$$
\item[(ii)] $\nabla_XY$ is $\mathbb R$-linear in $Y$: for $a_1,a_2\in\mathbb R$ and $Y_1,Y_2\in\Gamma(E),$ $$\nabla_X(a_1Y_1+a_2Y_2)=a_1\nabla_XY_1+a_2\nabla_XY_2.$$
\item[(iii)] $\nabla$ satisfies the following product rule: For $f\in A,$ $$\nabla_X(fY)=f\nabla_XY+(Xf)Y.$$
\end{itemize}
\end{definition}

\begin{proposition} Any vector bundle $\pi:E\rightarrow M$ over an associative variety $M$ has a connection $$\nabla:\mathfrak X(M)\times\Gamma(E)\rightarrow\Gamma(E).$$
\end{proposition}

\begin{proof} Define a connection $\nabla:\mathfrak X(M)\times\Gamma(E)\rightarrow\Gamma(E)$ on the open affine cover $\pi:E=\pts_{\mathbb R}B\rightarrow M=\pts_{\mathbb R}A$ by $$\overline\nabla_XY=\overline\nabla_X(y_1)dx_1+\cdots+\overline\nabla_X(y_n)dx_n\in\mathfrak X(M).$$ As $M$ and $E$ are associative varieties, they are formed by natural gluing of their local affine varieties, and as such, the local definition of the connection lifts to a connection on $M$ in $E.$
\end{proof}

We continue with the affine associative variety $M=\pts_{\mathbb R}A,$ $A$ a finitely generated $\mathbb R$-algebra. Given a curve $\gamma:\mathbb R=\pts_{\mathbb R}\mathbb R[t]\rightarrow M.$ A vector field along $\gamma$ is a continuous map $V:\mathbb R\rightarrow TM$ such that for every point $t\in\mathbb R,$ $V(t)\in T_{\gamma(t)}M.$ We denote the set of all smooth vector fields along $\gamma$ by $\mathfrak X(\gamma).$

\begin{lemma} $\mathfrak X(\gamma)$ is a real vector space which is also an $\mathbb R[t]$-module.
\end{lemma}

\begin{proof} $\mathfrak X(\gamma)=\hmm_{\mathbb R}(\ph A,\mathbb R[t]).$ This proves the claim.
\end{proof}

A curve $\gamma$ in $M$ corresponds to a homomorphism $\gamma:\ph A\rightarrow\mathbb R[t]$ such that when the generators of $A$ is $\{x_i\}_{i=1}^n,$ we can let $\gamma_i(t)=\gamma(x_i).$ Then as the generators of $T_{\gamma(t)}M$ is $\{dx_i\}_{i=1}^n,$ the expression $\gamma'(t)=\sum_{i=1}^n\gamma_i'(t)dx_i$ is well defined. 

\begin{definition} Let $\gamma:\mathbb R\rightarrow M$ be a curve in the associative variety $M.$ Then the velocity of $\gamma$ is defined as $\gamma'(t)\in T_{\gamma(t)}M.$
\end{definition}

The idea with all this, is that a curve is a straight line if its velocity vector does not change with $t.$

A vector field $V:\mathbb R\rightarrow TM$ along a curve $\gamma:\mathbb R\rightarrow M$ is extendible if there exists a vector field $\tilde V\in\mathfrak X(M)$ such that $V(t)=\tilde V(\gamma(t)).$ Writing this on the algebra level, this means that we have a commutative diagram \begin{equation}\label{liftdiag}\xymatrix{\ph A\ar[r]^V\ar[d]_{\tilde V}&\mathbb R[t]\\A.\ar[ur]_\gamma&}\end{equation}

This proves the following theorem in Lee, \cite{Lee18}.

\begin{theorem} Let $M$ be an associative variety with a connection  $$\nabla:\mathfrak X(M)\times\mathfrak X(M)\rightarrow\mathfrak X(M).$$ For each curve $\gamma:\pts_{\mathbb R}\rightarrow M,$ the connection $\nabla$ determines a unique operator $$D_t:\mathfrak X(\gamma)\rightarrow\mathfrak X(\gamma),$$ satisfying the following properties:
\begin{itemize} \item[(i)] Linearity over $\mathbb R$:
$$D_t(aV+bW)=aD_tV+bD_tW,\ a,b\in\mathbb R.$$
\item[(ii)] Product rule:
$$D_t(fV)=f'V+fD_tV,\ f\in\mathbb R[t].$$
\item[(iii)] If $V\in\mathfrak X(\gamma)$ is extendible, then for every extension $\tilde V$ of $V,$ $$D_tV(t)=\nabla_{\gamma'(t)}\tilde V.$$
\end{itemize}
\end{theorem}

\begin{definition} Let $M$ be an associative variety with a connection $\nabla$ in $TM.$ Let $\gamma:\pts_{\mathbb R}\mathbb R[t]\rightarrow M$ be a curve.
\begin{itemize}
\item[(ii)] The acceleration of $\gamma$ is the vector field $D_t\gamma'$ along $\gamma.$
\item[(ii)] We call $\gamma$ a geodesic with respect to $\nabla$ if $D_t\gamma'\equiv 0.$
\end{itemize}
\end{definition}

From the commutative diagram \ref{liftdiag}, we get our final result. This is Theorem 4.27 in Lee \cite{Lee18}.

\begin{theorem} Let $M$ be a smooth manifold and $\nabla$ a connection on $TM.$ For every $p\in M, w\in T_pM, t_0\in\mathbb R,$ there is a unique geodesic $\gamma:\mathbb R\rightarrow M$ satisfying $\gamma(t_0)=p$ and $\gamma'(t_0)=w.$
\end{theorem}

%For applications to Riemannian geometry, we are concerned with connections in the tangent bundle, that is $$\nabla:\mathfrak X(M)\times\mathfrak X(M)\rightarrow\mathfrak X(M).$$ On an open affine $M=\pts_{\mathbb R}A$ with $\mathfrak X(M)=\pts_{\mathbb R}(\ph(A))$ where $A$ is generated by $\{x_1,\dots,x_n\}$ the connection is determined by its connection coefficients defined by the relation $$\nabla_{x_i}x_j=\sum_{k=1}^n\Gamma^k_{ij}dx_k.$$


\begin{thebibliography}{99.}






\bibitem{Laudal21} O. A. Laudal, Mathematical models in science, World Scientific Publishing Co. Pte. Ltd., Hackensack, NJ, 2021

\bibitem{Lee18}
John M. Lee,
Introduction to Riemannian Manifolds,
Springer, Graduate Texts in Mathematics,
2018

\bibitem{S23}
Arvid Siqveland, 
Associative Algebraic Geometry,
World Scientific Publishing Co. Pte. Ltd., Hackensack, NJ, 2023
ISBN: 977-1-80061-354-6,
2023

\bibitem{S2603}
Arvid Siqveland,
Riemannian Geometry on Associative Varieties,
arXiv:2604.10462 [math.AG],
2026


\end{thebibliography}
\end{document}